\documentclass[11pt, a4paper]{amsart}
\usepackage[english]{babel}
\usepackage{amsfonts,amsmath}	
\usepackage{epsfig}
\usepackage{graphicx}		
\usepackage{amsthm,amssymb}
\usepackage[subnum]{cases}
\usepackage{url}
\usepackage{verbatim}
\usepackage{inputenc}
\usepackage{texdraw}
\usepackage{cite}
\makeatletter
\def\url@leostyle{%
  \@ifundefined{selectfont}{\def\UrlFont{\sf}}{\def\UrlFont{\small\ttfamily}}}
\makeatother
\input xy
\xyoption{all}				

\theoremstyle{definition}
\newtheorem{thm}{Theorem}

\newtheorem{lem}[thm]{Lemma}
\newtheorem{defi}[thm]{Definition}
\newtheorem{prop}[thm]{Proposition}
\newtheorem{exam}[thm]{Example}
\newtheorem{rem}[thm]{Remark}
\newtheorem{claim}[thm]{Claim}

\title{The gonality sequence of complete graphs}
\author{Filip Cools}
\email{{\tt f.cools@kuleuven.be}}
\address{KU Leuven, Department of Mathematics, Celestijnenlaan 200B, 3001 Heverlee, Belgium}
\author{Marta Panizzut}
\email{{\tt panizzut@math.tu-berlin@de}}
\address{TU Berlin, Institut f\"ur Mathematik, Stra{\ss}e des 17.Juni 136, 10623 Berlin, Germany}
\keywords{Gonality sequence, complete graphs, plane curves}
\subjclass[2010]{14T05, 14H51, 05C99}

\begin{document}

\begin{abstract} The gonality sequence $(\gamma_r)_{r\geq1}$ of a finite graph / metric graph / algebraic curve comprises the minimal degrees $\gamma_r$ of linear systems of rank $r$. For the complete graph $K_d$, we show that $\gamma_r =  kd - h$ if $r<g=\frac{(d-1)(d-2)}{2}$,
where $k$ and $h$ are the uniquely determined integers such that $r = \frac{k(k+3)}{2} - h$ with $1\leq k\leq d-3$ and $0 \leq h \leq k $. This shows that the graph $K_d$ has the gonality sequence of a smooth plane curve of degree $d$. The same result holds for the corresponding metric graphs. 
\end{abstract}

\maketitle

\begin{section}{Introduction}
Baker and Norine in \cite{BN} introduced a theory of linear systems on graphs, later generalized by several authors to metric graphs and other combinatorial objects \cite{MZ, GK}. It presents strong analogies with the one on algebraic curves.  Many remarkable theorems have been proven to have a combinatorial counterpart, for example the Riemann-Roch Theorem and Clifford's Theorem, see \cite{BN, MZ, Cop}.
%\vspace{\baselineskip}

With the notation $g^r_s$ we indicate a linear system of degree $s$ and rank $r$. We refer to Section \ref{0} for the definitions. The {\sl gonality sequence} $(\gamma_r)_{r\geq1}$ of a finite graph is defined as 
\[\gamma_r := \min\{ s \ \textrm{such that there exists a} \ g^r_s\}.\] 
The gonality sequence was first introduced in \cite{LM} in the context of algebraic curves and the terminology comes from the integer $\gamma_1$, which is called the {\sl gonality}. By Riemann-Roch, it follows that $\gamma_r = g+r$ if $r\geq g$. 

The gonality sequence is known for general chains of loops, as it follows from the results in \cite{CDPR}. Except for graphs with low gonality, the gonality sequence is still undetermined for other graphs.
The main result proven in this paper is the following: 

\begin{thm} \label{gs} The gonality sequence of $K_d$ is given by
\[
\gamma_r = \begin{cases}  kd - h &  \textrm{if} \ r <g \\ 
g+r &  \textrm{if} \ r \geq g 
\end{cases} \quad ,\\
\]
where $g=\frac{(d-1)(d-2)}{2}$ is the genus of $K_d$, and $k$ and $h$ are the uniquely determined integers with $1\leq k\leq d-3$ and $0 \leq h \leq k $ such that 
\[
r = \frac{k(k+3)}{2} - h.
\]
\end{thm} 
 
\vspace{\baselineskip} 

The proof consists of two parts, respectively presented in Section \ref{1} and in Section \ref{2}. 
In Section \ref{1} we verify the existence of divisors of degree $s=kd - h$ and rank (at least) $r=\frac{k(k+3)}{2} - h$. In other words, we prove that $kd - h$ is an upper bound for $\gamma_r$. Herefore, we use an algorithm, presented by Cori and Le Borgne in \cite{CLB}, for computing the rank of divisors on complete graphs. 
The more involving part is showing that the upper bound $kd - h$ for $\gamma_r$ is in fact sharp, which is done in Section \ref{2}. We translate this problem into a property of sequences of integers $(\alpha_i)_{i=1,\dots,d-1}$ satisfying certain hypotheses. In both parts, reduced divisors play an essential role. We provide an easy characterization of reduced divisors on complete graphs.   

In Section \ref{mg}, we extend the result to the complete metric graph $K_d$ with edge lengths equal to one. To be precise, we show that the gonality sequence of this metric graph is the same as the one of the corresponding ordinary graph; hereby answering \cite[Conjecture 3.14]{Bak} in the case of complete graphs. The existence part of the theorem follows immediately from \cite[Theorem 1.3]{HKN}, which relates the rank of divisors on graphs and their corresponding metric graphs. The second part is proven in an analogous way as in Section \ref{2}, again utilizing the property of integer sequences. 

We conclude by remarking that our arguments do not directly extend to complete metric graphs with arbitrary edge lengths. This is done in Section \ref{arbitrary}, where we give a quick account on the problems we encountered.

\vspace{\baselineskip} 

While the proof is purely combinatorial, the motivation for considering the question comes from plane curves and the specialization of divisors from curves to graphs. 

Baker in \cite{Bak} showed that there is a close connection between linear systems on curves and on graphs. Indeed, let $R$ be a discrete valuation ring with field of fractions $K$. A model for a curve $X$ over $K$ is a surface over $R$ whose generic fiber is $X$. Given a smooth curve $X$ over $K$ and a strongly semistable regular model $\mathfrak{X}$ over $R$, it is possible to specialize a divisor on the curve to a divisor on the dual graph of the special fiber of $\mathfrak{X}$. The complete graph $K_d$ pops up if one takes a model of a smooth plane curve of degree $d$ degenerating to a union of $d$ lines. 

The gonality sequence of smooth plane curves has been computed by Ciliberto in \cite{Cil}, and Hartshorne in \cite{Har}. Therefore a natural question in this setting is whether the complete graph and the smooth plane curve have the same gonality sequence. Theorem \ref{gs} provides a positive answer. As we point out in Remark \ref{specialization}, the linear systems with degree $s=dk -h$ and rank $r=\frac{k(k+3)}{2} -h$ on plane curves specialize to linear systems of the same degree and rank on the graphs. So the first part of Theorem \ref{gs} can in fact be deduced from the result on plane curves. 

\begin{subsection}*{Acknowledgments} We thank Marc Coppens for the suggestion of the problem. We are very grateful for his valuable comments on early drafts of the paper. This work was conducted in the framework of Research Project G.0939.13N of the Research Foundation - Flanders (FWO).
\end{subsection}

\end{section}

\begin{section}{Linear systems and reduced divisors} \label{0}

Let $G$ be a finite graph without loop edges. We denote its vertex set by $V(G)$ and its edge set by $E(G)$.  
\begin{defi} A {\sl divisor} on $G$ is an element of the free abelian group $\textrm{Div}(G)$ on the set $V(G)$. Any divisor can be represented in a unique way as a finite formal combination of vertices of $G$ with integer coefficients:
\[
D = \sum_{v \in G} a_v \, (v), \ \ \textrm{with} \ a_v \in \mathbb{Z}. 
\]
The {\sl degree} $\textrm{deg}(D)$ of a divisor $D$ is the sum $\sum_{v \in V(G)} a_v$ of its coefficients. If $a_v \geq 0$ for every $v \in G$, the divisor is said to be {\sl effective}, and this is indicated with $D \geq 0$. 

\vspace{\baselineskip}

Let $f: V(G) \rightarrow \mathbb{Z}$ be an integer-valued function on the vertices of $G$. We define the {\sl principal divisor} corresponding to $f$ as 
\[ 
 \textrm{div}(f) = \sum_{v \in V(G)} \Big( \sum_{e=vw \in E(G)} \big(f(v) - f(w) \big)\Big) (v) .
\] 

Two divisor $D_1, D_2 \in \textrm{Div}(\Gamma)$ are {\sl linearly equivalent}, denoted by $D_1 \sim D_2$, if there exists a function $f$ such that 
\[
D_1  - D_2 = \textrm{div}(f). 
\]

\vspace{\baselineskip}

The {\sl linear system} corresponding to $D$, indicated with $|D|$, is the set of the effective divisors linearly equivalent to $D$. In symbols,
\[
|D| = \Big\{ E \in \textrm{Div}(G)\, \Big| \ E \geq 0, \, E \sim D\Big\}.
\]

The {\sl rank} $\textrm{rk}_{G}(D)$ of a divisor $D$ is defined as $-1$ if $D$ is not equivalent to any effective divisor, otherwise
\[
\textrm{rk}_{G}(D)= \max \Big\{ r \in \mathbb{Z}_{\geq 0} \Big| \ |D-E| \not = \emptyset \ \ \  \forall \ E \in \textrm{Div}(G), \ E \geq 0, \ \textrm{deg}(E)=r\Big\}. 
\]
We will often omit the subscript $G$ in $\textrm{rk}_{G}(D)$ when it is clear from the context on which graph $G$ we are working. 
\end{defi}

\medskip

Reduced divisors will play an important role in proving the results of this paper. We briefly repeat the definition presented in \cite{BN}. 
\begin{defi} Let $A$ be a subset of $V(G)$. Given $v \in A$, the {\sl outgoing degree $\textrm{outdeg}_A(v)$ of $A$ at $v$} is defined as the number of edges having $v$ as one endpoint and whose other endpoint lies in $V(G) \setminus A$. Let $D$ be a divisor on $G$ and $v$ be a vertex of $G$. The divisor $D$ is {\sl $v$-reduced} if it is effective in $V(G) \setminus \{v\}$ and each non-empty subset $A \subseteq V(G)\setminus \{ v \}$ contains a point $w$ such that $D(w) < \textrm{outdeg}_A(w)$.
\end{defi}
The following result is proven in \cite[Proposition 3.1]{BN}:

\begin{prop}
Let $v$ be a vertex on a graph $G$. Then for every divisor $D$ on $G$, there exists a unique $v$-reduced divisor $D'$ such that 
$D'\sim D$. 
\end{prop}

Let $K_d$ be the complete graph on $d$ vertices. 

\begin{lem} \label{reduced}
 A divisor $D$ on $K_d$ is reduced with respect to a vertex $v$ if and only if there exists an ordering $v_1, v_2, \dots, v_{d-1}$ of the vertices in $V(K_d) \setminus \{v\}$ such that $0 \leq D(v_{i}) \leq i-1$. 
\end{lem}

\begin{proof}
For every $A \subseteq V(K_d) \setminus \{v\}$ and for every $w \in A$, we have that $\textrm{outdeg}_A(w) = d-|A|$. Therefore  $D$ is $v$-reduced if and only if each non-empty subset $A$ contains an element $w$ such that $D(w) < d-|A|$. 

Suppose that there exists an ordering such that $D(v_i) \leq i-1$. Let $A$ be a subset of $V(K_d) \setminus \{v\}$, and let $j = \min \{i \, | \, v_i \in A\}$. We have that $j \leq d-|A|$. It follows that 
\[D(v_j) \leq j-1 < d-|A|, 
\]
therefore $D$ is $v$-reduced.  

%\vspace{\baselineskip} 

For the other implication, if we take $A = V(K_d) \setminus \{v\}$, then there exists a vertex $v_1$ in $A$ such that $D(v_1) =0$. For $A = V(K_d) \setminus \{v, v_1\}$, there exists a vertex $v_2$ such that $D(v_2) \leq 1$. Iterating this argument, we obtain that if $A = V(K_d) \setminus \{v_0, \dots, v_i\}$, there exists a vertex $v_{i+1}$ in $A$ such that $D(v_{i+1}) \leq i$. 
\end{proof}

\begin{rem} \label{existDv0}
In particular, if $D$ is a $v$-reduced divisor on $K_d$, then there exists a vertex $v_1\neq v$ such that $D(v_1)=0$. 
\end{rem}
\end{section}

\begin{section}{An upper bound for the gonality sequence} \label{1}

We provide a completely combinatorial proof of the inequality $\gamma_r\leq kd - h$. Since 
$$\text{rk}(D-(v))\geq r(D)-1$$ for each divisor $D$ and each vertex $v$ (see \cite[Lemma 2.7]{Bak}), it suffices to handle the case $h=0$, i.e.\ there exists a divisor $D$ of degree $kd$ and rank $r=\frac{k(k+3)}{2}$. We will explicitly construct such a divisor: if $V(K_d)=\{v_1,\ldots,v_d\}$, then the divisor 
 \[
  D = kd \, (v_d) \sim \sum_{i=1}^d \, k(v_i) 
 \]
will do the job. 

%From this and the definition of rank it also follows that, for every $1\leq h\leq k$, there exist divisors of degree $kd-h$ and rank 
%$\frac{k(k+3)}{2} -h$. 

\vspace{\baselineskip}

In \cite{CLB} the authors provide an algorithm to compute the rank of a divisor on the complete graph $K_d$. Although it is stated in a different terminology, it is possible to present it in terms of divisors. 

\begin{lem} [Lemma 3 in \cite{CLB}] \label{rank}
Let $D$ be a $v_d$-reduced effective divisor on $K_d$. Let $v_i\neq v_d$ be a vertex for which $D(v_i) = 0$ (see Remark \ref{existDv0}). The divisor $D' = D - (v_i)$ has rank rk$(D) - 1$. 
\end{lem}

The algorithm takes a divisor as input. The first step consists of computing the $v_d$-reduced divisor $D$ equivalent to it. If $D(v_d) <0$, then the divisor has rank $-1$. Otherwise we take a vertex $v_i$ such that $D(v_i) = 0$ and consider the divisor $D_1 = D - (v_i)$. By the previous lemma rk$(D) = \textrm{rk}(D_1) + 1$. We iterate until we find a divisor of rank $-1$. The algorithm terminates after at most 
deg$(D)$ steps. 

\begin{thm}
The rank of the divisor $D = kd\, (v_d)$ is $\frac{k(k+3)}{2}$.
\end{thm}

\begin{proof}
 We use Lemma \ref{rank} and the algorithm described above. For convenience, we write the subsequent divisors appearing in the algorithm as $D_{s,t}$ where $s\geq 1$ and $0\leq t\leq s$. In this way, we obtain a sequence of divisors where the indexes are in lexicographic order:  
\[ D_{1,0}, \, D_{1,1},\, D_{2,0}, \, D_{2,1},\, D_{2,2}, \, D_{3,0}, \, \dots
\] The first two divisors are:
 \[
 \begin{array}{ccl}
  %D_{0,0} = D = kd\, (v_d)
  D_{1,0} = D - (v_1) &\sim& \Big[kd - (d-1)\Big](v_d) + \sum_{i=2}^{d-1} (v_i)\\
  D_{1,1} = D_{1,0} - (v_1) &\sim& \Big[kd - (d-1)-1\Big](v_d)+ (d-2)(v_1)\\
  \end{array}
 \]
At every step we subtract the divisor $(v_{i})$ corresponding to the vertex with zero coefficient and smallest index $i$. So,
\[
D_{2,0} = D_{1,1} - (v_2) \sim \Big[kd - 2(d-1)\Big](v_d) + \sum_{i=3}^{d-1}2(v_i) + (v_2).
\]
% The next ones: 
 % \[
 % \begin{array}{rrl}
 % D_{2,0} = D_{1,1} - (v_2) &\sim& [kd - 2(d-1)](v_d) + 2(v_{d-1}) + \cdots + 2(v_3) + (v_2)\\
 % D_{2,1} = D_{2,0} - (v_1) &\sim& [kd - 2(d-1)-1](v_d) + \sum_{i=3}^{d-1}(v_i) + (d-2)(v_1)\\
 % D_{2,2} = D_{2,1} - (v_2) &\sim& [kd - 2(d-1)-2](v_d) +  (d-2)(v_2) + (d-3)(v_1)\\
 % \end{array}
 %\]
At the step $(s,t)$ (with $s\leq d-1$), the divisor $D_{s,t}$ is linearly equivalent to the following $v_d$-reduced divisor: 
\[ 
  \Big[kd - s(d-1) - t\Big](v_d) + \sum_{i=s+1}^{d-1}(s-t)(v_i) + \sum_{i=t+1}^{s} (i-t-1)(v_i) + \sum_{i=1}^t(d-2-t+i)(v_i).
\]
%\begin{align*}
%& \Big[kd - s(d-1) - t\Big](v_d) + \sum_{i=s+1}^{d-1}(s-t)(v_i) + \sum_{i=t+1}^{s} (i-t-1)(v_i) \\ & + \sum_{i=1}^t(d-1-i)(v_i) \\
%& = \Big[kd - s(d-1) - t\Big](v_d) + 0(v_{t+1}) + 1(v_{t+2}) + \ldots + [s-t-1](v_s) \\
%& \hphantom{=} + [s-t](v_{s+1}) + \ldots + [s-t](v_{d-1}) + [d-t-1](v_t) +\ldots + [d-2](v_1) 
%\end{align*}
For $(l,j)=(k,k)$, we have 
\[
 D_{k,k} \sim (d-2)(v_k) + \cdots + (d-k-1)(v_1)
\]
The divisor obtained at the next step is 
\[
 D_{k+1,0}=D_{k,k} - (v_{k+1}) = (-1)(v_{k+1}) + (d-2)(v_k) + \cdots + (d-k-1)(v_1).
\]
Since it is $v_{k+1}$-reduced, it is not equivalent to any effective divisor, so it has rank $-1$. We can conclude that the rank of $D$ is 
$\frac{(k+1)(k+2)}{2}-1 = \frac{k(k+3)}{2}$.
\end{proof}

\begin{rem} \label{specialization} For the reader familiar with the theory of specialization of divisors from curves to graphs, we remark that it is possible to prove the first part of Theorem \ref{gs} using the gonality sequence of smooth plane curves \cite{Cil} and Baker's Specialization Lemma \cite{Bak}. In fact, the complete graph $K_d$ is the dual graph of the special fiber of a regular strongly semistable model of a smooth plane curve of degree $d$. In \cite{Cil}, Ciliberto showed that, given an effective divisor $E'$ of degree $h\leq k$ on $X$ and $H$ the generic divisor cut out by a line, the complete linear system 
\[
|kH -E'|
\]
has rank $\frac{k(k+3)}{2} - h$. 

The divisors $kH-E'$ described above specialize to divisors of the form 
\[
D = (k-a_1)(v_1) + (k-a_2)(v_2) + \cdots + (k-a_d)(v_d),
\]
on $K_d$, with $a_1 + a_2 + \cdots + a_d = h$. By the Specialization Lemma, the rank of these divisors is at least $\frac{k(k+3)}{2}-h$.  

It is not difficult to prove that it cannot be strictly bigger. By a relabeling of the vertices, we can assume that $a_i=0$ for $i>h$ and $a_i \geq a_j$ if $i < j$. We consider the effective divisor 
\[
E = \sum_{i=1}^{k+1} \Big(k - a_i -(i-2)\Big) (v_i)
\]
of degree $\frac{k(k+3)}{2} - h +1$. We have that 
\[ 
D- E = -1(v_1) + 0(v_2) + (v_3) + \cdots  + (k-1)(v_{k+1}) + k(v_{k+2}) + \cdots + k(v_d).
\]
This divisor is $v_1$-reduced, therefore $|D- E| = \emptyset$ and the rank of $D$ is $\frac{k(k+3)}{2} - h$.
\end{rem}
\end{section}

\begin{section}{Sharpness of the upper bound}\label{2}

This section is concerned with proving the inequality $\gamma_r\geq kd-h$ for $r=\frac{k(k+3)}{2}-h$, with $k \leq d-3$ and $0 \leq h \leq k$. In fact, we show that each divisor $D$ of degree $dk - h-1$ has rank strictly smaller than $\frac{k(k+3)}{2}-h$. It suffices to do this for $h=k$. This follows from 
$\textrm{rk}(D+(v)) \leq \textrm{rk}(D) +1$ for each divisor $D$ and vertex $v$. 

\bigskip

Let $D$ be any divisor of degree $k(d-1)-1$ on $K_d$. By the above argumentation, we need to show that its rank is strictly smaller than 
$$\frac{k(k+3)}{2}-k = \frac{k(k+1)}{2}.$$ 
We may assume that $D$ is reduced with respect to the vertex $v_d$. By Lemma \ref{reduced}, we can label the vertices so that the coefficients satisfy the following inequalities:
\[
D(v_1) = 0 \leq D(v_2) \leq \cdots \leq D(v_{d-1})\leq d-2 \ \ \textrm{and} \ \ D(v_i) \leq i-1 \ \textrm{for} \ i\leq d-1.
\]

If $D(v_d) < \frac{k(k+1)}{2}$, we can already conclude that the rank of the divisor cannot be $\frac{k(k+1)}{2}$, so assume that $D(v_d) \geq \frac{k(k+1)}{2}$. 

We write 
\[
D(v_d)=a(d-1) + b \quad \text{with}\quad  a,b\in\mathbb{Z}_{\geq 0} \ \text{and}\ 0 \leq b \leq d-2.  
\]
Since $D(v_d)\leq \deg(D)=k(d-1)-1$, we have that $a<k$. 
Consider the divisor
\[
D' = \Big( b- (d-1) \Big) (v_d) + \sum_{i=1}^{d-1} \Big(D(v_i) + a+1\Big) (v_i) \sim D. 
\]
Note that $D'$ has a negative coefficient at $v_d$. If we are able to construct an effective divisor $E$ of degree at most $\frac{k(k+1)}{2}$ which is supported on the vertices $v_1, \dots, v_{d-1}$ and such that $D'-E$ is $v_d$-reduced, then we can conclude that the rank of $D$ cannot be $\frac{k(k+1)}{2}$. By Lemma \ref{reduced}, this happens if  
\[
\sum_{i=1}^{d-1} \max\Big\{0, D(v_i) +a -(i-2)\Big\} \leq \frac{k(k+1)}{2}. 
\]
To simplify the notation, we define 
\[\alpha_i:= D(v_i) +a -(i-2) \ \textrm{and} \ \alpha_i^+:=\max\Big\{0, D(v_i) +a -(i-2)\Big\}.\] 

We have that
\begin{align*}
\sum_{i=1}^{d-1} \alpha_i 
&= \Big[\deg(D)-D(v_d)\Big]+a(d-1)-\sum_{i=1}^{d-1}\,(i-2) \\
&= k(d-1) - \frac{(d-2)(d-3)}{2} - b.
\end{align*}
From the above arguments, it follows that the following inequalities are satisfied for every index $i \leq d-1$:
\begin{equation}  
\alpha_i \leq a + 1, \ \ \alpha_i \geq a-i+2, \ \ \alpha_{i+1} \geq \alpha_i -1. \tag{$\ast$} \label{condition}
\end{equation}
In particular, we can deduce that $\alpha_1 = a+1$ and $\alpha_{a+2} \geq 0$. 

\vspace{\baselineskip}

The inequality $\sum_{i=1}^{d-1} \alpha_i^+ \leq \frac{k(k+1)}{2}$ is not always satisfied, as the following example illustrates: 
\begin{exam}\label{example1}
Consider the $v_d$-reduced divisor 
\[D=\Big((k-1)(d-1)\Big)(v_d)+ (0(v_1)) + \sum_{i=2}^{d-1} (v_i).
\]
In this case, $a=k-1$, $b=0$ and  
\[
D'= -(d-1)(v_d) +  k(v_1) + \sum_{i=2}^{d-1}(k+1)(v_i) \sim D.
\]
Furthermore, we have $\alpha_1= k$ and $\alpha_i = k-(i-2)$ for $i \geq 2$, so 
\[
\sum_{i=1}^{d-1} \alpha_i^+ = k + \frac{k(k+1)}{2} > \frac{k(k+1)}{2}.
\]
It is still possible to show that the divisor does not have rank $\frac{k(k+1)}{2}$, by considering instead the divisor 
\[
D''= (k-1)(v_1)+ \sum_{i=2}^{d-1}k(v_i)  \sim D. 
\]
We remark that in this case the coefficient at $v_d$ is not negative. Let $E$ be the following divisor of degree $\frac{k(k+1)}{2}$:
\[
E = (v_d)+ (k-1)(v_1)+ \sum_{i=2}^{k+1} (k-i+1)(v_i).  
\]
The divisor $D'-E$ has a negative coefficient at the vertex $v_d$ and is $v_d$-reduced. Therefore, also for this example we can conclude that $D$ has not rank $\frac{k(k+1)}{2}$. 
\end{exam}

We can generalize Example \ref{example1} as follows: if $t_1=\sum_{i=1}^{d-1} \alpha_i^+ > \frac{k(k+1)}{2}$, instead of the divisor $D'$, we consider  
\[
D'' = b (v_d) + \sum_{i=1}^{d-1} \Big(D(v_i) + a\Big) (v_i) \sim D. 
\]
If we are able to construct an effective divisor $E$ of degree at most $\frac{k(k+1)}{2}$ with coefficient $b+1$ at $v_d$ and such that $D''-E$ is $v_d$-reduced, or equivalently, if 
\[
t_2=b+1+ \sum_{i=1}^{d-1} \max \Big \{0, \alpha_i-1 \Big \} \leq \frac{k(k+1)}{2},   
\]
then we can conclude that the rank of $D$ cannot be $\frac{k(k+1)}{2}$.  

We claim that (at least) one of the two terms $t_1$ and $t_2$ is at most $\frac{k(k+1)}{2}$.  

\begin{claim} \label{statement2} Let $\alpha=(\alpha_i)_{i=1, \dots, d-1}$ be a sequence of integers satisfying the rules in (\ref{condition}) and such that
\[
\sum_{i=1}^{d-1} \alpha_i = k(d-1) - \frac{(d-2)(d-1)}{2} -b.
\]
Let $\alpha_i^+ = \max \{0, \, \alpha_i\}$ and $(\alpha_i-1)^+ = \max \{0, \, \alpha_i-1\}$. 
Define 
\[
t_1 := \sum_{i=1}^{d-1} \alpha_i^+ \ \ \textrm{and} \ \ t_2 := b+1 + \sum_{i=1}^{d-1} \Big(\alpha_i-1\Big)^+. 
\]
Then $\min\{t_1, t_2\}\leq \frac{k(k+1)}{2}$. 
\end{claim}

\begin{rem}
To summarize our strategy, we add the principal divisor $$-(d-1)(v_d)+\sum_{i=1}^{d-1} (v_i)$$ to $D$ until it either has a negative value at $v_d$, or one time before that, and (at least) one of the two resulting divisors is $v_d$-reduced after subtracting an effective divisor of degree $\frac{k(k+1)}{2}$.
\end{rem}
\bigskip

The idea of proof of Claim \ref{statement2} is as follows: first we introduce a specific integer sequence, which we show to be the ``worst-case scenario''. Afterwards, we prove the claim for this particular sequence. 

For each $p \in \{1, \dots, d-2\}$ and $q \in\{p+2,\ldots, d\}$, we define the sequence $\beta^{(p,q)}=\big(\beta_i\big)_{i=1, \dots, d-1}$ as follows: 
\begin{gather} \nonumber
\beta_1 = \cdots = \beta_p = a+1, \\ \tag{$\triangle$}  \label{sequence}
\beta_{p+1} = a, \ \ldots, \ \beta_{q-1}=p+a+2-q, \ \beta_{q} = p+a+2-q, \\
\beta_{q+1} = p+a+1-q,\ \ldots,\ \beta_{d-1}=p+a-d+3. \nonumber
\end{gather}
In case $p=d-1$ we are slightly abusing the notation, since there is no admissible value for $q$. The sequence $\beta^{(d-1, q)}$ becomes $\beta^{(d-1)} = (a+1, a+1, \dots, a+1)$. 

If $q=d$, then $\beta_{d-1}=p+a-d+2$. For example, the sequence $\beta^{(1,d)}$ is given by $(a+1, a, a-1, \dots, a-d+3)$. 

Note that $\beta^{(p,q)}$ satisfies the rules (\ref{condition}). 

\begin{rem}
 Let $\alpha=(\alpha_i)_{i=1, \dots, d-1}$ be a sequence of integers satisfying (\ref{condition}). 
 For every $i$ we have the following inequalities:
 \[
  a-i+2=\beta^{(1,d)}_i\leq \alpha_i \leq \beta^{(d-1)}_i = a+1
 \]
 Moreover for $q > p+2$ it holds that 
 \[ 
\sum_{i=1}^{d-1} \beta^{(p,q-1)}_i = \sum_{i=1}^{d-1}\Big( \beta^{(p,q)}_i \Big) +1.
\]
Similarly, for $p <d-2$ we have
\[
 \sum_{i=1}^{d-1} \beta^{(p+1,d)}_i = \sum_{i=1}^{d-1} \Big( \beta^{(p,p+2)}_i \Big) +1, 
\]
and for $p=d-1$ 
\[
 \sum_{i=1}^{d-1} \beta^{(d-1)}_i = \sum_{i=1}^{d-1} \Big( \beta^{(d-2,d)}_i \Big) +1.
\]

 It follows that for each sequence $\alpha$ satisfying (\ref{condition}), there exists a unique sequence $\beta^{(p,q)}$ with $\sum_{i=1}^{d-1} \beta_i=\sum_{i=1}^{d-1} \alpha_i$.
 
 \end{rem}

\begin{rem}\label{remark}

From the formula 
\[
\sum_{i=1}^{d-1} \beta_i = k(d-1) - \frac{(d-2)(d-3)}{2} -b, 
\]
it is possible to compute $p$ and $q$. Indeed, since
\[
\sum_{i=1}^{d-1} \beta_i = p(a+1) + \frac{a(a+1)}{2} - \frac{(d-a-p-2)(d-a-p-1)}{2} + (d-q), 
\]
it follows that 
\[
q = (d-1)\Big[p-(k-a)\Big] - \frac{p(p-1)}{2} +(b+2). 
\]
Since the difference between two subsequent terms (corresponding to $p+1$ and $p$) is 
\[
(d-1) - \frac{p(p+1)}{2} + \frac{(p-1)p}{2} = (d-1)-p, 
\]
there is a unique $p$ for which $q \in \{p+2, \dots, d\}$. Note that $p\geq k-a$, since $q \geq p+2$.
\end{rem}

\begin{lem} \label{lemma}
Let $k,d,a,b$ be integers such that $$d-2\geq k+1\geq 2, \quad 0\leq a\leq k-1\quad \text{and}\quad 0\leq b \leq d-2.$$ 
Let $\alpha=(\alpha_i)_{i=1,\ldots,d-1}$ be a sequence of integer numbers that satisfies the conditions $\alpha_i \leq a+1$, $\alpha_{i+1} \geq \alpha_i -1$ and $$\sum_{i=1}^{d-1} \alpha_i = k(d-1) - \frac{(d-2)(d-3)}{2} -b.$$ If $p,q$ are integers such that the sequence $\beta^{p,q}$ satisfies $\sum_{i=1}^{d-1} \beta_i=\sum_{i=1}^{d-1} \alpha_i$, 
then $$\sum_{i=1}^{d-1} \alpha_i^+ \leq \sum_{i=1}^{d-1} \beta_i^+ \quad \text{and} \quad
\sum_{i=1}^{d-1} (\alpha_i-1)^+ \leq \sum_{i=1}^{d-1} (\beta_i-1)^+.$$
\end{lem}

\begin{proof}
%We write $(\beta_i)_{i=1,\ldots,d-1}$ for the specific sequence described above, omitting the apex $(p,q)$. Let $(\alpha_i)_{i=1,\ldots,d-1}$ be any other sequence satisfying $\alpha_i \leq a+1$ and $\alpha_{i+1} \geq \alpha_i -1$. 
Consider sequences $\alpha$ and $\beta^{(p,q)}$ that satisfy the conditions in the statement. By construction, the sequence $\beta^{(p,q)}=(\beta_i)_{i=1,\ldots,d-1}$ is taken in such a way that it minimizes the number of places $i$ with positive values for $\beta_i$, fixing the value for $\sum_{i=1}^{d-1} \beta_i^+$. 

Suppose that $q < a+p+2$, hence $\beta_q>0$. Let $i_1,\ldots,i_c$ be the indexes $i$, ordered from small to big, for which $\alpha_i$ has negative values, and $j_1,\ldots,j_{d-a-p-3}$ the same for $\beta_i$ (so $j_h=h+a+p+2$). Then $d-a-p-3 \geq c$ since the $\beta_i$ are negative in as much places as possible. Moreover, we have that $\alpha_{i_h} \geq \beta_{j_h}$ for each $h\in\{1,\ldots,c\}$. Indeed, if $\alpha_{i_h} < \beta_{j_h} = -h$, then $\alpha_{i_1} < -1$ by repeatedly using the rule $\alpha_{i+1} \geq \alpha_i -1$, hence $\alpha_{i_1-1}<0$, a contradiction.  
Therefore, the sum of the negative $\alpha_i$ is at least the sum of the negative $\beta_i$. Since $\sum_{i=1}^{d-1} \beta_i = \sum_{i=1}^{d-1}\alpha_i$, we can conclude that 
\[\sum_{i=1}^{d-1}\,\beta_i^+\geq \sum_{i=1}^{d-1}\,\alpha_i^+.
\]
 Moreover, 
\[
 \sum_{i=1}^{d-1} (\beta_i-1)^+ = \sum_{i=1}^{d-1} \beta_i^+ - (a+p+1) \geq \sum_{i=1}^{d-1} \alpha_i^+ - (a+p+1) \geq  \sum_{i=1}^{d-1} (\alpha_i-1)^+.
\]

Now suppose that $q \geq a+p+2$, so $\beta_q\leq 0$ (if $q\neq d$). The number of $\beta_i \geq 0$ is $a+p+1$. If $\sum_{i=1}^{d-1} \alpha_i^+ > \sum_{i=1}^{d-1} \beta_i^+$, then the number of places $i$ where $\alpha _i \geq 0$ is more than $a+p$. From this, it follows that  
the sum of the negative $\alpha_i$ is at least the sum of the negative $\beta_i$. We obtain a contradiction with $\sum_{i=1}^{d-1}\alpha_i = \sum_{i=1}^{d-1} \beta_i$. Analogously, we can see that $\sum_{i=1}^{d-1}(\alpha_i-1)^+ \leq \sum_{i=1}^{d-1}(\beta_i -1)^+$.
\end{proof}

We conclude this section by proving Claim \ref{statement2}.

\begin{proof}
Because of Lemma \ref{lemma}, the sequence $\beta^{(p,q)}=(\beta_i)_{i =1, \ldots, d-1}$ defined by (\ref{sequence}) maximizes $\sum_{i=1}^{d-1} \beta_i^+$ and $b+1+\sum_{i=1}^{d-1} (\beta_i -1)^+$. Hence it suffices to check the claim for these kind of sequences. 
We distinguish three cases: 
\begin{enumerate} 
\item $\beta_{d-1} >0$; \label{case1}
\item $\beta_{d-1} \leq 0$ and $q < a+p+2$; \label{case2} 
\item $\beta_{d-1} \leq 0$ and $q \geq a+p+2$. \label{case3}
\end{enumerate}
 
\vspace{\baselineskip}
{\bf Case (\ref{case1})}. Remark that $\beta_i^+ = \beta_i$ for every $i$. We compute $t_2$:   
\begin{align*}
t_2&= \sum_{i=1}^{d-1} \Big(\beta_i -1\Big)^+ + b+1 = \sum_{i=1}^{d-1} \beta_i^+ - (d-1) + b+1 \\
& = \sum_{i=1}^{d-1} \beta_i - (d-1) + b+1 
%= \Big[ k(d-1) - \frac{(d-2)(d-3)}{2} -b\Big] - (d-1) + b+1 
= (k-1)(d-1) - \frac{(d-2)(d-3)}{2} + 1
\end{align*}
It follows that $t_2 \leq 	\frac{k(k+1)}{2}$ if and only if 
\[
\frac{(d-k-1)(d-k-2)}{2} \geq 0. 
\]
This condition is satisfied since $k \leq d-3$.  

\vspace{\baselineskip}

{\bf Case (\ref{case2})}. Remark that $\beta_q \geq 1$. We compute $t_2$ also in this case: 
\begin{align*}
t_2&= \sum_{i=1}^{d-1} \Big(\beta_i -1\Big)^+ + b+1 = \sum_{i=1}^{d-1} \beta_i^+ - (a+p+1) + b+1 \\
&= \sum_{i=1}^{d-1} \beta_i + \frac{(d-p-a-3)(d-p-a-2)}{2} - (a+p) + b \\
&= \Big[ (k-1)(d-1) - \frac{(d-2)(d-3)}{2} +1\Big] + \frac{(d-p-a-2)(d-p-a-1)}{2}
\end{align*}
It follows that $t_2 \leq 	\frac{k(k+1)}{2}$ if and only if 
\[
\frac{(d-p-a-1)(d-p-a-2)}{2} \leq \frac{(d-k-1)(d-k-2)}{2}. 
\]
This is satisfied since $p \geq k-a$.  

\vspace{\baselineskip} 

{\bf Case (\ref{case3})}. First we consider the special situation $p = k-a$. We compute the term $t_1$: 
\begin{align*}
t_1&= \sum_{i=1}^{d-1} \beta_i^+ = p(a+1) + \frac{a(a+1)}{2} = \frac{(p+a)(p+a+1)}{2} - \frac{p(p-1)}{2} \\
&= \frac{k(k+1)}{2} - \frac{p(p-1)}{2} \leq \frac{k(k+1)}{2}.
\end{align*}

If $p > k-a$, we compute $t_2$:  
\begin{align*}
t_2&= \sum_{i=1}^{d-1} \Big(\beta_i -1\Big)^+ + b+1 = \sum_{i=1}^{d-1} \beta_i^+ - (a+p) + b+1 \\
&= \sum_{i=1}^{d-1} \beta_i + \frac{(d-p-a-1)(d-p-a-2)}{2} - (d-q) - (a+p) + b+1 \\
&= \Big[(k-1)(d-1)- \frac{(d-2)(d-3)}{2} +1\Big] + \frac{(d-p-a-1)(d-p-a-2)}{2} \\
& \ \ \ + q -p-a-1.
\end{align*}
It follows that $t_2 \leq 	\frac{k(k+1)}{2}$ if and only if 
\[
\frac{(d-p-a-1)(d-p-a-2)}{2} + q -p-a-1\geq \frac{(d-k-1)(d-k-2)}{2}. 
\]
This condition is satisfied since $p > k-a$.  
\end{proof}
\end{section}

\begin{section}{Metric graphs} \label{mg}
If $G$ is a graph, we define the {\sl metric graph} $\Gamma$ {\sl corresponding} to $G$ as the metric graph with $V(\Gamma) = V(G)$, $E(\Gamma) = E(G)$ and edge lengths $l(e)=1$, see \cite{HKN}. In this section we explain how to extend the previous results to the complete metric graphs corresponding to the graphs $K_d$, which we will also denote by $K_d$. We start by briefly recalling the main definitions regarding linear systems on metric graphs. We refer to \cite{MZ, GK, HKN} for further reading. 

\begin{defi} Let $\Gamma$ be a metric graph. A {\sl divisor} on $\Gamma$ is an element of the free abelian group $\textrm{Div}(\Gamma)$ on the points of the graph, so
\[
D = \sum_{p \in \Gamma} a_p\,(p) \ \ \textrm{with} \ a_p \in \mathbb{Z}. 
\]
The {\sl degree} of $D$, denoted by $\textrm{deg}(D)$, is the sum of its coefficients. %$\textrm{deg}(D) = \sum_{p \in \Gamma} a_p$. 
As before, if $a_p \geq 0$ for every $p \in \Gamma$, the divisor is said to be {\sl effective}. The {\sl support of a divisor $D$} is the set of points $p$ of $\Gamma$ such that $a_p \not =0$ and it is indicated with $\textrm{supp}(D)$.

\vspace{\baselineskip}

A {\sl rational function} $f:\Gamma \rightarrow \mathbb{R}$ on $\Gamma$ is continuous, piecewise linear with integer slopes and only finitely many pieces. The {\sl principal divisor} $\textrm{div}(f)$ associated to $f$ is the divisor whose coefficient at $p$ is given by the sum of the incoming slopes of $f$ at $p$. Only at a finite number of points, the coefficients are not zero.

Two divisors $D_1$, $D_2 \in \textrm{Div}(\Gamma)$ are {\sl linearly equivalent}, $D_1 \sim D_2$, if there exists a rational function $f$ such that 
\[
D_1  - D_2 = \textrm{div}(f). 
\]

\vspace{\baselineskip}

The {\sl linear system} of $D$, indicated with $|D|$, is the set of the effective divisors linearly equivalent to $D$, 
\[
|D| = \Big\{ E \in \textrm{Div}(\Gamma)\Big| \ E \geq 0, E \sim D\Big\}.
\]
The {\sl rank $\textrm{rk}_{\Gamma}(D)$ of a divisor} is defined as $-1$ if $D$ is not equivalent to any effective divisor, otherwise
\[
\textrm{rk}_{\Gamma}(D)= \max \Big\{ r \in \mathbb{Z}_{\geq 0}\Big| \ |D-E| \not = \emptyset \ \ \  \forall \ E \in \textrm{Div}(\Gamma), \ E \geq 0, \ \textrm{deg}(E)=r\Big\}. 
\]
\end{defi}

\vspace{\baselineskip}

Again, reduced divisor on metric graphs will play an important role. We recall the definition from \cite{HKN} and \cite{Luo}. 

\begin{defi}
Let $\Gamma$ be a metric graph and $X$ be a closed connected subset of $\Gamma$. Given $p \in \partial X$, the {\sl outgoing degree $\textrm{outdeg}_X(p)$ of $X$ at $p$} is defined as the maximum number of internally disjoint segments in $\Gamma \setminus X$ with an open end in $p$. Let $D$ be a divisor on $\Gamma$. A boundary point $p \in \partial X$ is {\sl saturated with respect to $X$ and $D$} if $D(p) \geq \textrm{outdeg}_X(p)$, and {\sl non-saturated} otherwise. A divisor $D$ is {\sl $p$-reduced} if it is effective in $\Gamma \setminus \{p\}$ and each closed connected subset $X \subseteq \Gamma \setminus \{ p \}$ contains a non-saturated boundary point.
\end{defi}

\begin{thm}[Proposition 7 in \cite{MZ}]
Let $D$ be a divisor on a metric graph $\Gamma$. For every point $p \in \Gamma$ there exists a unique $p$-reduced divisor linearly equivalent to $D$. 
\end{thm}

Let $K_d$ be the complete metric graphs on $d$ vertices. We present a similar characterization of reduced divisors on $K_d$ as in Lemma \ref{reduced}. Given a vertex $v$ of $K_d$ and an ordering $v_1, \dots, v_{d-1}$ of the vertices in $V(K_d) \setminus \{v\}$, for every $1\leq i \leq d-1$ we define 
\[
A_i = (v,v_i] \sqcup \bigsqcup_{j< i} \ (v_j,v_i].
\]
Remark that $K_d \setminus \{v\} = \sqcup_{i=1}^{d-1} A_i$. See Figure \ref{figure} for an example.
\begin{figure}[h]
\centering
\includegraphics[width=4cm,height=4cm,keepaspectratio]{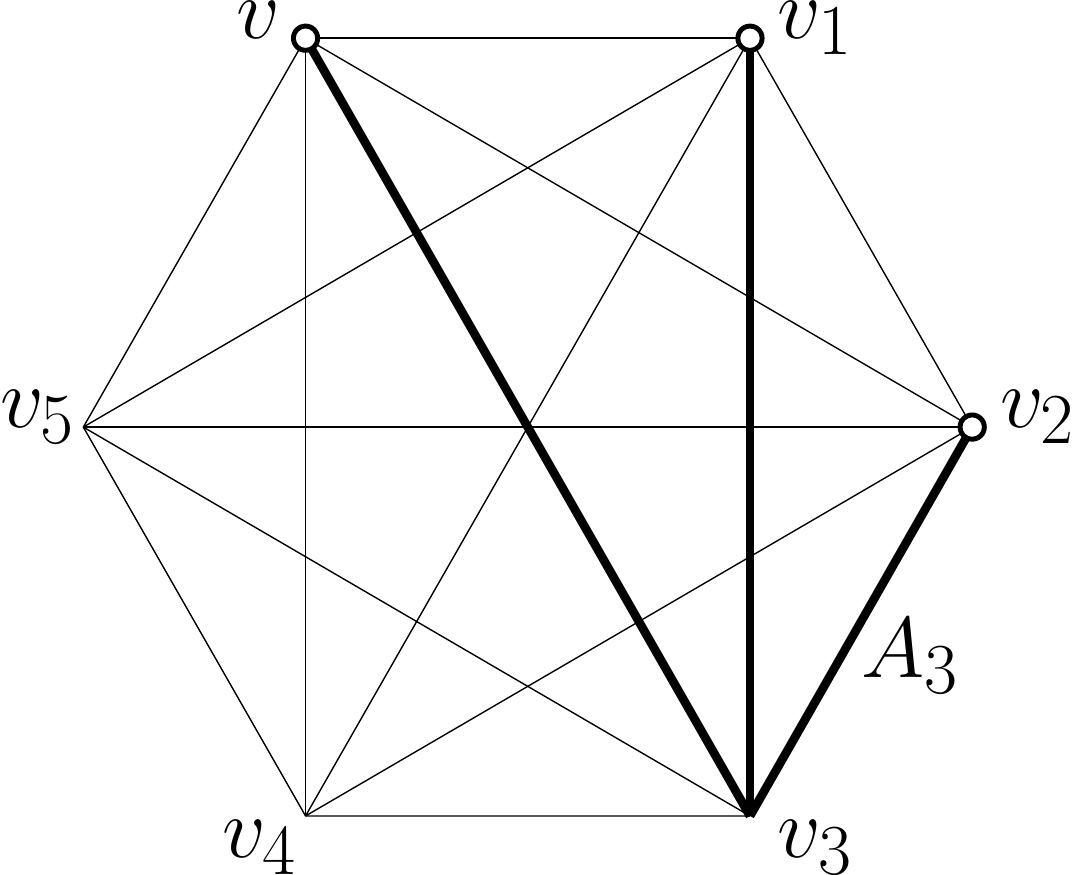} 
\caption{The thickened edges represent the set $A_3$. The vertices $v$, $v_1$ and $v_2$ are not contained in $A_3$.} 
\label{figure}

\end{figure}

\begin{lem} \label{reduced1} Let $D$ be a divisor on the metric graph $K_d$. It is reduced with respect to a vertex $v$ if and only if the following conditions are satisfied: 
\begin{enumerate}
\item $D$ is effective in $K_d \setminus \{v\}$; \label{c1}
\item for every edge $e \in E(K_d)$, \label{c2}
\[
\sum_{p \in e^{\circ}} D(p) \leq 1;
\]
\item \label{c3}
there exists an ordering $v_1, \dots, v_{d-1}$ of the vertices in $V(K_d) \setminus \{v\}$ such that  
\[
\mathfrak{D}(v_i) := \textrm{deg}(D|_{A_i}) \leq i-1.
\]
\end{enumerate}
\end{lem} 
\begin{proof} We start by proving the `if' part, so assume that $K_d$ satisfies the three conditions. We need to show that every closed connected subset $X$ of $K_d \setminus \{v\}$ contains a non-saturated boundary point. If $X$ contains a boundary point $p$ that is not in the support of $D$, than $p$ is non-saturated. So we focus on the case that the boundary points of $X$ are all in the support of $D$. In case $X =\{p\}$, with $p$ a point in the interior of an edge, then $p$ is non-saturated because of condition (\ref{c2}). 

Suppose that $X$ satisfies the hypothesis that if two vertices are contained in $X$, then also the edge connecting them is contained in $X$. We define 
\[
A : = V(K_d) \cap X \ \ \textrm{and} \ \ A^{\complement} = (K_d \setminus X) \cap V(K_d) = V(K_d) \setminus A. 
\]
A vertex $w$ of $X$ is not-saturated if and only if 
\[
D(w) + \Big|\Big\{q \in \partial X \,:\, q \in (w, w') \ \textrm{with} \ w' \in A^{\complement}\Big\}  \Big| < d - |A|. 
\]
Let $j = \min\{i \ | \ v_i \in A\}$, so $j\leq d-|A|$. By condition (\ref{c3}), we know that $\mathfrak{D}(v_j) \leq j-1$. We show that $v_j$ is non-saturated:
\begin{align*}
& D(v_j) + \Big|\Big\{q \in \partial X \,:\, q \in (v_j, w') \ \textrm{with} \ w' \in A^{\complement} \Big\}  \Big| \\
& = \mathfrak{D}(v_j) + \Big|\Big\{q \in \partial X \,:\, q \in (v_j, v_i) \ \textrm{with} \ i > j \ \textrm{and} \ v_i \in A^{\complement} \Big\}  \Big|  \\
& \leq \mathfrak{D}(v_j) + d - |A| - j  \\
& \leq j-1 + d - |A| - j = d-|A|-1.
\end{align*}
If the hypothesis on the closed subset $X$ is not satisfied, consider a closed connected subset $X'\subset K_d\setminus\{v\}$ containing $X$ and that satisfies the hypothesis. For each vertex $v \in X$ it holds 
that $\textrm{outdeg}_{X'}(v) \leq \textrm{outdeg}_X(v)$. Therefore if $v$ is non-saturated for $X'$, it is also non-saturated for $X$. 
\vspace{\baselineskip}

For the `only if' part, suppose that $D$ is $v$-reduced. It is clear that the conditions (\ref{c1}) and (\ref{c2}) are satisfied. We now show condition (\ref{c3}). As a first step, we claim that there exists a vertex $v_1\neq v$ such that $$(v,v_1]\cap \text{Supp}(D)=\emptyset.$$ Indeed, if not, there is a point $p\in (v,w]\cap \text{Supp}(D)$ for each vertex $w\neq v$, and we can consider a connected subset $X\subset K_d\setminus\{v\}$ for which the boundary consists of all these points $p$. Note that $X$ contains the complete subgraph on the vertices in $V(K_d) \setminus \{v\}$. Since each boundary point $p\in\partial X$ is saturated, this is in contradiction with the fact that $D$ is $v$-reduced. 

In a similar way, we can find a vertex $v_2\neq v,v_1$ such that the sum of the coefficients $D(p)$ with $p\in (v,v_2]\cup(v_1,v_2]$ is at most $1$. In order to show this, we work with a connected component $X$ for which the boundary consists of points $p\in \text{Supp}(D)$ lying on either $(v,w]$ or $(v_1,w]$ for some vertex $w\neq v,v_1$. This component contains the complete subgraph on the vertices in $V(K_d) \setminus \{v,v_1\}$.    

Iterating this argument, we find back condition (\ref{c3}).
   
%we first take $X$ to be the complete subgraph on the vertices in $V(K_d) \setminus \{v\}$. Then there exists a vertex $v_1$ in $X$ such that $\mathfrak{D}(v_1) =0$. For $X$ the complete subgraph on the vertices in $V(K_d) \setminus \{v,v_1\}$, there exists a vertex $v_2$ such that $\textrm{deg}(D|_{A_1)}\leq1$. This implies $\mathfrak{D}(v_2) \leq 1$. Iterating this we obtain that if $X$ is the complete subgraph on the vertices in $V(K_d) \setminus \{v,v_1, \dots v_{i-1}\}$, then there exists a vertex $v_i$ in $X$ such that $\mathfrak{D}(v_{i+1}) \leq i$. 
\end{proof}
\begin{exam} Consider the graph and the divisor $D$ drawn in Figure \ref{figes}. The divisor is $v$-reduced, since we have 
\[
\mathfrak{D}(v_1) =0, \ \ \mathfrak{D}(v_2) =0, \ \ \mathfrak{D}(v_3) =2, \ \ \mathfrak{D}(v_4) =2, \ \ \text{and} \ \ \mathfrak{D}(v_5) =4.
\]
\begin{figure}[h]
\centering
\includegraphics[width=4cm,height=4cm,keepaspectratio]{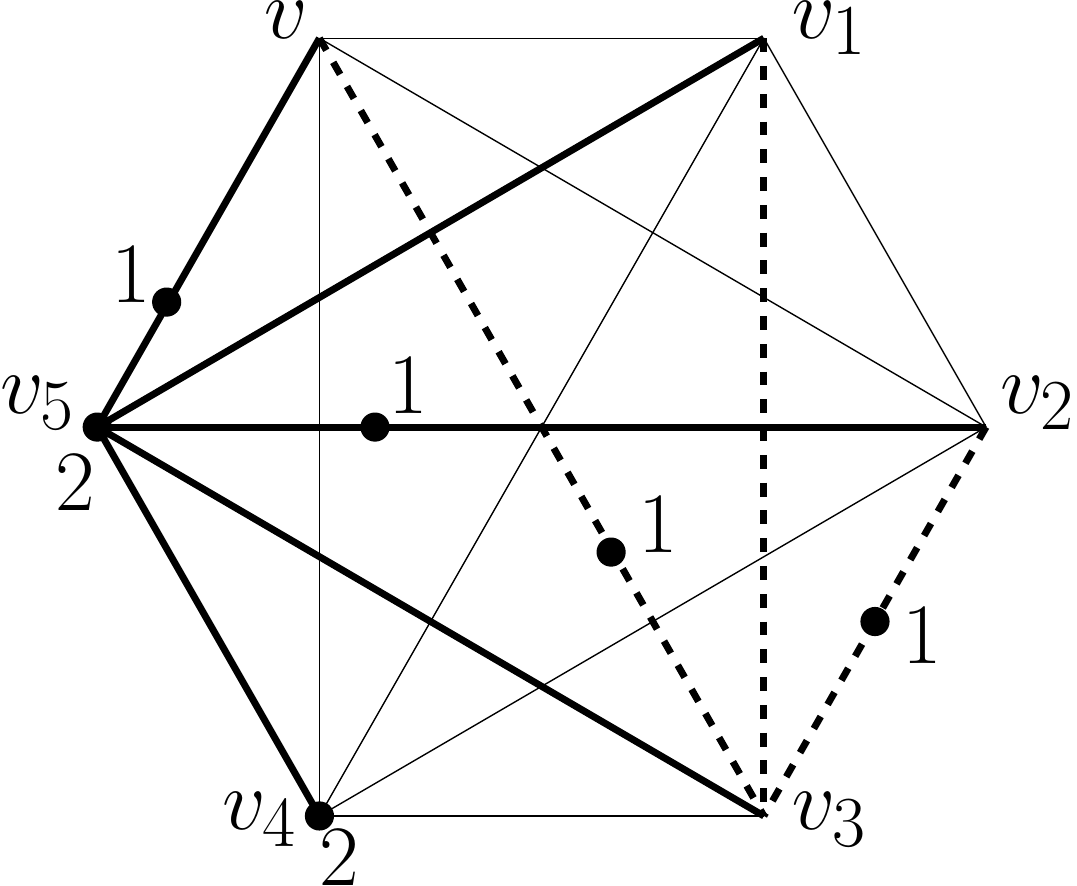} 
\caption{The thickened edges represent the set $A_5$, and the thickened and dashed ones the set $A_5$.} 
\label{figes}
\end{figure}
\end{exam}

The definitions of linear systems $g^r_s$ and gonality sequences translate to the setting of metric graphs. We prove that the gonality sequence of complete metric graphs $K_d$ is the same as the one of the ordinary complete graphs. 

\begin{thm} \label{gsmetric} The gonality sequence of the metric graph $K_d$ is 
\[
\gamma_r = \begin{cases}  kd - h &  \textrm{if} \ r <g \\ 
g+r &  \textrm{if} \ r \geq g. 
\end{cases} \\
\]
where $k$ and $h$ are the uniquely determined integers with $1\leq k\leq d-3$ and $0 \leq h \leq k $ such that 
\[
r = \frac{k(k+3)}{2} - h.
\]
\end{thm} 

In the proof of the theorem we will use the following result from \cite{HKN}, which relates the rank on graphs with the rank on metric graphs: 
\begin{thm}\label{ranghi}
Let $D$ be a divisor on a loopless graph $G$ and let $\Gamma$ be the metric graph corresponding to $G$. Then, 
\[
\textrm{rk}_G(D) = \textrm{rk}_{\Gamma}(D).
\]
\end{thm}

\begin{proof}[Proof of Theorem \ref{gsmetric}]
Again for every $k\geq 1$ and for every $0 \leq h \leq k$ there are two statements that need to be shown: 
\begin{itemize} 
\item there exists a divisor of degree $kd - h$ and rank $\frac{k(k+3)}{2} - h$; 
\item there does not exist a divisor of degree strictly smaller than $kd- h$ and rank $\frac{k(k+3)}{2} - h$. 
\end{itemize}  

The first statement follows directly from Theorem \ref{gs} and Theorem \ref{ranghi}. We consider the second statement. As before, it is enough to set $h=k$ and to show that the rank of every divisor $D$ of degree $dk-k-1$ is strictly smaller than $\frac{k(k+3)}{2}-k = \frac{k(k+1)}{2}$.

\vspace{\baselineskip}

Let $D$ be $v_d$-reduced. By Lemma \ref{reduced1}, we can label the vertices so that the following inequalities are satisfied:
\[
\mathfrak{D}(v_1) = 0 \leq \mathfrak{D}(v_2) \leq \cdots \leq \mathfrak{D}(v_{d-1})\leq d-2, \ \ \textrm{and} \ \ \mathfrak{D}(v_i) \leq i-1, \ \textrm{for} \ i\leq d-1.
\]

We write 
\[D(v_d) = a(d-1) + b, \ \textrm{with} \ a,b \in \mathbb{Z}_{\geq0}, \ 0 \leq b \leq d-2 \ \textrm{and} \ a\leq k-1. 
\]
Consider the divisors
\begin{align}\label{prob1}
&D' = \Big( b- (d-1) \Big) (v_d) + \sum_{i=1}^{d-1} \Big[D|_{A_i} + (a+1) (v_i)\Big] \sim D, \\
&D'' =  b \, (v_d) + \sum_{i=1}^{d-1} \Big[D|_{A_i} + a \, (v_i)\Big] \sim D. \label{prob2}
\end{align}
We define $\mathfrak{a}_i := \mathfrak{D}(v_i) + a - (i-2)$ and $\mathfrak{a}_i^+ := \max \{0, \mathfrak{a}_i\}$.
The sequence $\mathfrak{a}=(\mathfrak{a}_i)_{i=1,\ldots,d-1}$ will play the role of the sequence $\alpha$ in Section \ref{2}. 
In particular, we have that
\[
\sum_{i=1}^{d-1} \mathfrak{a}_i = k(d-1) - \frac{(d-2)(d-3)}{2} - b. 
\]
and the inequalities in (\ref{condition}) are satisfied.

Because of Claim \ref{statement2}, either
\[
\mathfrak{t}_1 := \sum_{i=1}^{d-1} \mathfrak{a}_i^+ \ \ \textrm{or} \ \ \mathfrak{t}_2 := b+1 + \sum_{i=1}^{d-1} \Big(\mathfrak{a}_i-1\Big)^+ 
\]
is at most $\frac{k(k+1)}{2}$. This allows us to construct an effective divisor $E$ of degree $\frac{k(k+1)}{2}$ such that either $D'-E$ or $D''-E$ is $v_d$-reduced and has a negative coefficient at $v_d$, hence $\text{rk}(D)\leq \frac{k(k+1)}{2}$. 
\end{proof}
\end{section}

\begin{section}{What about arbitrary edge lengths?}\label{arbitrary}

In this section, we will focus on complete metric graphs with arbitrary edge lengths, which we will denote by $K_d(\ell)$ for a vector 
$\ell\in(\mathbb{R}_{>0})^{d \choose 2}$. The arguments presented in the previous sections do not directly extend to the graphs $K_d(\ell)$. Therefore, computing the gonality sequence of $K_d(\ell)$ is still an open problem in general. Below, we briefly explain the obstructions we encountered. 
\vspace{\baselineskip}

In Section \ref{1}, we proved that the divisor $D = \sum_{i=1}^{d}k(v_i)$ on the graph $G=K_d$ has rank 
$\frac{k(k+3)}{2}$ for $1\leq k\leq d-3$. Theorem \ref{ranghi} allowed us to conclude that the divisor $D$ viewed on the corresponding metric graph $\Gamma=K_d$ has the same rank. This result does not hold if the lengths of the edges are arbitrary. Moreover, our computation of the rank on $G=K_d$ relies on \cite[Lemma 3]{CLB}, which uses the symmetries of complete graphs. 

Even though we can not use the arguments from Section \ref{1}, we still expect that the divisor $D = \sum_{i=1}^{d} k(v_i)$ on 
$\Gamma=K_d(\ell)$ has rank $\frac{k(k+3)}{2}$, if $1\leq k\leq d-3$.  
This statement holds for $k=d-3$: in this case, the divisor $D$ coincides with the canonical divisor $K_\Gamma$, which has rank 
$g-1=\frac{d(d-3)}{2}$. It is an easy exercise to check this statement for $k=1$ and (by using Riemann-Roch) $k=d-4$. The following results provide extra motivation.   

\begin{lem} 
The divisor $D = \sum_{i=1}^d k(v_i)$ on $\Gamma=K_d(\ell)$ has rank at most $\frac{k(k+3)}{2}$. 
\end{lem}

\begin{proof} Consider the effective divisor 
\[
E = \sum_{j=1}^{k+2}(k+2-j)(v_j), 
\]
which has degree $\frac{k(k+3)}{2}+1$. The divisor $D-E$ is $v_1$-reduced (Lemma \ref{reduced1} also holds for arbitrary edge lengths) and has a negative coefficient at $v_1$, hence $|D-E|=\emptyset$
and $\text{rk}_\Gamma(D)\leq \frac{k(k+3)}{2}$. 
\end{proof}

\begin{prop}
If $d\geq 5$, the divisor $D = \sum_{i=1}^{d} 2(v_i)$ on $\Gamma=K_d(\ell)$ has rank $5$.
\end{prop}

\begin{proof} We need to show that $|D-E| \not = \emptyset$ for every effective divisor $E$ of degree $5$. We may assume that $E$ is supported on the vertices, since the vertex set is rank-determining (see \cite{Luo}). 
If $E(v_i)\leq 2$ for all $i$, then $D-E\geq 0$ and the statement follows directly. Hence, we may assume that $E(v_i)\geq 3$ for some index $i$, so $E$ is of the form $3(v_i) + (v_j) +(v_k)$, $3(v_i) + 2(v_j)$, $4(v_i)+(v_j)$ or $5(v_i)$ for some $i\neq j\neq k \neq i$. We will only handle the latter case; the others are proven in an analogous manner.

%If $E = 3(v_i) + (v_j) +(v_k)$, with $i \not = j$, $j\not =k$ and $i \not =k$, then we can immediately conclude. 
Suppose that $E = 5(v_i)$. Denote by $\delta$ the minimal length of the edges $(v_i,v_j)$ with $j\neq i$ and by $p_j \in (v_i,v_j)$ the point on distance $\delta$ from $v_j$. Then  
\[
D-E=-3(v_i) + \sum_{j \not = i} 2(v_j) \quad \sim \quad F:=-3(v_i) + \sum_{j \not = i} 2(p_j).
\]
If $p_j = v_i$ for more than one index $j$, then we are done since $F\geq 0$. Otherwise, this means that the divisor $F$ has coefficient $-1$ at $v_i$. Say that $k$ is the unique index for which $p_k = v_i$. Now we consider the linearly equivalent divisor 
\[
F':=-(v_i) + \sum_{j \not =i, \ j \not = k} \Big((s_j)+(t_j)\Big),\]
with $s_j, t_j \in (v_j,v_i)$ the points on equal and maximal distance from $p_j$, where $s_j$ is in between $p_j$ and $v_j$, and $t_j$ is in between $p_j$ and $v_i$. Because of the maximality assumption, either $t_j=v_i$ or $s_j=v_j$ for all $j\neq i$. If $t_j = v_i$ for at least one $j$, we can conclude since $F'\geq 0$. Otherwise, we have that $s_j = v_j$ for all $j$. Then we can consider a linearly equivalent divisor  
\[
F''=-(v_i) + \sum_{j\not =i, \ j\not=k} \Big((s'_j)+(t'_j)\Big) 
\]
with $t'_j \in (v_j,v_i)$, $s'_j$ on the path $(v_j,v_k,v_i)$ and $t'_j = v_i$ or $s'_j = v_i$ for a certain $j$, hence $F''\geq 0$. 
\end{proof}

\vspace{\baselineskip}

Also the proof of the sharpness of the upper bound cannot directly be extended to the graphs $\Gamma=K_d(\ell)$. Indeed, the divisor $D'$ and $D''$ introduced in (\ref{prob1}) and (\ref{prob2}) are no longer linearly equivalent to $D$. It is natural to consider instead the following divisors 
\begin{align}
&D' = \Big( b- (d-1) \Big) (v_d) + \sum_{i=1}^{d-1} \Big[D|_{A_i} + (a+1) (p_i)\Big] \sim D, \\
&D'' =  b \, (v_d) + \sum_{i=1}^{d-1} \Big[D|_{A_i} + a \, (p_i)\Big] \sim D,
\end{align}
with $p_i \in (v_d,v_i)$ and $p_j = v_j$ for at least one index $j$. But then we run into problems while constructing the effective divisor $E$ of degree $\frac{k(k+1)}{2}$, since in general $D'-E$ and $D''-E$ are not $v_d$-reduced. 
\end{section}

\bibliographystyle{amsalpha}
\bibliography{bibliografia}

\providecommand{\bysame}{\leavevmode\hbox to3em{\hrulefill}\thinspace}
\providecommand{\MR}{\relax\ifhmode\unskip\space\fi MR }
\providecommand{\MRhref}[2]{%
  \href{http://www.ams.org/mathscinet-getitem?mr=#1}{#2}
}
\providecommand{\href}[2]{#2}
\begin{thebibliography}{CDPR11}

\bibitem[Bak08]{Bak}
M.~Baker, \emph{Specialization of linear systems from curves to graphs},
  Algebra \& Number Theory \textbf{2} (2008), no.~6, 613--653.

\bibitem[BN07]{BN}
M.~Baker and S.~Norine, \emph{{R}iemann - {R}och and {A}bel - {J}acobi theory
  on a finite graph}, Adv. Math. \textbf{215} (2007), no.~2, 766--788.

\bibitem[CB16]{CLB}
R.~Cori and Y.~Le Borgne, \emph{On computation of {B}aker and {N}orine's rank
  on complete graphs}, Electron. J. Combin. \textbf{23} (2016), no.~1, Paper
  1.31, 47.

\bibitem[CDPR11]{CDPR}
F.~Cools, J.~Draisma, S.~Payne, and E.~Robeva, \emph{A tropical proof of the
  {B}rill-{N}oether {T}heorem}, Adv. Math \textbf{230} (2011), 759--776.

\bibitem[Cil83]{Cil}
C.~Ciliberto, \emph{Alcune applicazioni di un classico procedimento di
  {C}astelnuovo}, Sem. di Geom., Dipart. di Matem., Univ. di Bologna
  (1982-1983), 17--43.

\bibitem[Cop16]{Cop}
M.~Coppens, \emph{{C}lifford's theorem for graphs}, Adv. Geom. \textbf{16}
  (2016), no.~3, 389--400.

\bibitem[GK08]{GK}
A.~Gathmann and M.~Kerber, \emph{A {R}iemann-{R}och theorem in tropical
  geometry}, Mathematische Z. \textbf{259} (2008), no.~1, 217--230.

\bibitem[Har86]{Har}
Robin Hartshorne, \emph{Generalized divisors on {G}orenstein curves and a
  theoren of {N}oether}, J. Math. Kyoto Univ. \textbf{26} (1986), no.~3,
  375--386.

\bibitem[HKN13]{HKN}
J.~Hladk\'y, D.~Kr\'al', and S.~Norine, \emph{Rank of divisors on tropical
  curves}, J. of Combin. Theory Ser. {A} \textbf{120} (2013), no.~7,
  1521--1538.

\bibitem[LM12]{LM}
H.~Lange and G.~Martens, \emph{On the gonality sequence of an algebraic curve},
  Manuscripta Math. \textbf{137} (2012), no.~3-4, 457--473.

\bibitem[Luo11]{Luo}
Y.~Luo, \emph{Rank-determining sets of metric graphs}, J. of Combin. Theory
  Ser. {A} \textbf{118} (2011), 1755--1793.

\bibitem[MZ08]{MZ}
G.~Mikhalkin and I.~Zharkov, \emph{Tropical curves, their {J}acobians and
  {T}heta functions}, Curves and Abelian Varieties, Contemp. Math., vol. 465,
  Amer. Math. Soc., Providence, RI, 2008, pp.~203--230.

\end{thebibliography}

\end{document}